\documentclass[ECP,preprint]{ejpecp}
\pdfoutput=1 

\usepackage[utf8]{inputenc}
\usepackage[T1]{fontenc}

\usepackage[notextcomp]{stix2}
\newcommand{\independent}{\Vbar}

 \usepackage{bm} 

\usepackage[draft]{fixme}
\FXRegisterAuthor{pa}{apa}{PAZ}
\FXRegisterAuthor{jf}{ajf}{JFD}
\fxusetheme{colorsig}

\usepackage{caption}
\usepackage[margin=5pt,
justification=centering,
labelformat=parens,
]{subcaption}

\usepackage[mode=image]{standalone} 

\usepackage{pgfplots}
\usepgfplotslibrary{colorbrewer}
\pgfplotsset{compat=newest}
\pgfplotsset{colormap/Paired-6}

\newsavebox{\largestimage}

\usepackage[backend=biber,
	    url=false,
	    doi=false,
	    isbn=false,
	    giveninits=true,
	    date=year,
	    maxbibnames=99,
	    sortcites=true,
	    sorting=nty,
	    style=numeric]{biblatex}
\renewbibmacro{in:}{}

\usepackage{booktabs}
\usepackage{multirow}

\usepackage[shortlabels]{enumitem}

\newcommand{\DeclareAlphabet}[2]{
  \foreach \x in {A,B,...,Z}{%
    \expandafter\xdef
    \csname #1\x\endcsname{%
      \noexpand#2{\x}}%
  }
}

\DeclareAlphabet{d}{\mathbb}  
\DeclareAlphabet{c}{\mathcal}
 \DeclareAlphabet{b}{\mathbf}

 \let\R\dR
 \let\N\dN

\newcommand{\cb}{\ensuremath{\mathscr{B}}}
\newcommand{\ce}{\ensuremath{\mathscr{E}}}
\newcommand{\cf}{\ensuremath{\mathscr{F}}}
\newcommand{\cg}{\ensuremath{\mathscr{G}}}

\newcommand{\ci}{\ensuremath{\mathscr{I}}}

\newcommand{\norm}[1]{\left\lVert\,#1\,\right\rVert}
\newcommand{\rd}{\mathrm{d}}
\newcommand{\un}{\mathbb{1}}
\newcommand{\zero}{\mathbb{0}}

 \renewcommand{\P}{\dP}
 \newcommand{\E}{\ensuremath{\mathbb{E}}}
 \newcommand{\esp}[1]{\mathbb{E}\left[#1\right]}
 \newcommand{\prb}[1]{\dP\left[#1\right]}
 \newcommand{\given}{\middle|}
 
 \newcommand{\ind}[1]{\mathbb{1}_{#1}}
 \let\F\cF

\newcommand{\cl}{\ensuremath{\mathscr{L}^\infty}}

\newcommand{\cpa}{\mathcal{P}^\mathrm{Anti}}
\newcommand{\AF}{\cF^\mathrm{Anti}}
 \newcommand{\FF}{\ensuremath{\mathbf{F}}}

\newcommand{\kk}{\mathrm{k}}
\newcommand{\kkk}{{\bm{\kk}}}
 
\newcommand{\I}{\mathfrak{I}}
\newcommand{\loss}{\mathrm{L}}
\newcommand{\costad}{\mathfrak{c}}

\newcommand{\costu}{C_\mathrm{uni}}
\newcommand{\param}{\mathrm{Param}}

 \newcommand{\Tinf}{\ensuremath{\mathcal{T}}}

\newcommand{\ca}{\ensuremath{\mathscr{A}}}

\newcommand{\traits}{\mathcal{X}}

\newcommand{\common}{\mathcal{C}}

\newcommand{\extended}{\bm{\pi}}

\SHORTTITLE{%
  Equivalence by coupling for  heterogeneous SIS models
}

\TITLE{%
  Equivalence by coupling for  heterogeneous SIS models
  \support{This work is partially supported by Labex Bézout reference ANR-10-LABX-58}%
}

\AUTHORS{%
  Jean-François Delmas%
\footnote{%
  CERMICS, \'{E}cole nationale des ponts et chauss\'ees, 77455
  Marne-la-Vall\'ee, France.
  \EMAIL{jean-francois.delmas,dylan.dronnier@enpc.fr}}
\and
Dylan Dronnier%
\footnotemark[2]
\and
Pierre-André Zitt%
\footnote{LAMA, Université Gustave Eiffel, 77420
Champs-sur-Marne, France. 
\EMAIL{pierre-andre.zitt@univ-eiffel.fr}}%
}


\AMSSUBJ{92D30;58E17} 

\ABSTRACT{%
    We consider the optimal allocation of (perfect) vaccine in an
heterogeneous SIS model.  Using a coupling approach, we explain how
different models for the heterogeneity of the population lead to the
same Pareto frontier in the cost/loss valuation of the vaccinations
strategies. This covers in particular the elementary continuous representation of
discrete models and the measure preserving transformation which
appears in graphon theory.
}

\begin{document}

\section{Introduction}

We consider the effect of vaccination in an heterogeneous SIS model
(with S=Susceptible and I=Infectious), in the framework introduced in
\cite{delmas_infinite-dimensional_2020}.  The model, which will be
recalled in detail below, is parametrized by four elements: a feature
space, denoted by $\traits$ ; two real-valued functions $\gamma$ and
$\costad$ on $\traits$, representing the feature-dependent recovery
rate and  vaccination cost; a real-valued
function $k$ on $\traits^2$, encoding the infection rate between
individuals of different features.  We focus on
optimizing feature-dependent vaccination strategies, as discussed in
\cite{ddz-theory-optim}.

In classical probability theory, the same random experiment
may be represented by two different probability spaces and
random variables, with the same distribution. 
Unsurprisingly, the same situation occurs here in the
choice of the trait space and the associated parameters.
The goal of this article is to
  describe precisely a notion of equivalence between models via a
  coupling, and to compare equilibria and  optimal vaccination strategies
   between equivalent models. 

  We address the three following questions, see Theorem~\ref{thr:main}
  and Corollary~\ref{cor:=frontier}:
  \begin{enumerate}
  \item Do equivalent models lead to comparable optimal vaccination strategies?
    Is knowing the optima for one model enough to find the optima in
    equivalent models?
  \item If the feature space is ``too rich'', and encodes features that
    are not relevant to the propagation of the epidemic, is it possible
    to reduce the model by ``forgetting''  irrelevant features?
  \item Do equivalent models evolve in the same way, and in particular
    can we compare their equilibria?
  \end{enumerate}

  In the next section, we introduce the necessary notation,
  borrowing heavily from the presentation of~\cite{ddz-hit}.
  The main result is stated in
  Section~\ref{sec:equivalent} and gives positive answers to the three
  questions; the proofs are postponed to Section~\ref{sec:proof-coupling}.  
 Detailed examples are discussed in Section~\ref{sec:exple-couple}.

\section{Framework and notation}
\label{sec:main}

\subsection{The heterogeneous SIS model}
\label{sec:SIS-heterogeneous}
We recall the differential equations governing the epidemic dynamics
in meta-population~SIS models introduced in~%
\cite{delmas_infinite-dimensional_2020}, to which we refer for
additional context and details. 

Let~$(\traits, \cf,  \mu)$ be a  probability space, where~$x  \in \traits$
represents  a feature  and the  probability measure~$\mu(\mathrm{d}  x)$
represents  the  fraction  of  the  population  with  feature~$x$.   The
parameters of the SIS model are given by a \emph{recovery rate function}
$\gamma$, which  is a  positive bounded  measurable function  defined on
$\traits$, and a \emph{transmission rate kernel} $k$, where a kernel is a
nonnegative measurable  function defined on $\traits^2$.
In accordance with \cite{delmas_infinite-dimensional_2020}, we consider for a kernel $\kk$
on $\traits$ and $q\in (1, +\infty )$ its   norm:
$
  \norm{\kk}_{\infty ,q}=\sup_{x\in \traits}\, \left(\int_\traits \kk(x,y)^q\,
  \mu(\mathrm{d} y)\right)^{1/q}.
$ 
 For a kernel $\kk$ on $\traits$ such that $\norm{\kk}_{\infty ,q}$ is finite for some $q\in
(1, +\infty )$, we define the integral operator~$\Tinf_{\kk}$  on the set $\cl$ of bounded
measurable real-valued function on $\traits$ by:
\begin{equation*}
  \Tinf_\kk (g) (x) = \int_\traits \kk(x,y) g(y)\,\mu(\mathrm{d}y)
  \quad \text{for } g\in \cl \text{ and } x\in \traits.
\end{equation*}

By convention, for~$f,g$ two nonnegative measurable functions defined on~$\traits$
and~$\kk$ a kernel on~$\traits$, we denote by $f\kk g$ the kernel on $\traits$ defined by:
\begin{equation}
  \label{eq:def-fkg}
  f\kk g:(x,y)\mapsto f(x)\, \kk(x,y) g(y).
\end{equation}
We shall consider the kernel $\kkk=k\gamma^{-1}$, 
which is thus  defined by:
\begin{equation*}
  \boxed{    \kkk(x,y)=k(x,y)\, \gamma(y)^{-1}.}
\end{equation*}
We  assume that:
\begin{equation}
  \label{eq:bded}
  { \norm{\kkk}_{\infty ,q}<\infty 
    \quad\text{for some $q\in (1, +\infty )$.}}
\end{equation}
The integral operator~$\Tinf_{\kkk}$  is   the  so  called  \emph{next-generation
operator}.

\medskip

Let $\Delta=\{f\in \cl\,\colon\, 0\leq  f\leq 1\}$ be the
subset of nonnegative functions bounded by~$1$, and let $\zero, \un\in \Delta$
be  the constant  functions  equal  respectively to 0 and to 1.   The  SIS dynamics  considered
in~\cite{delmas_infinite-dimensional_2020} follows  the vector field~$F$
defined on~$\Delta$ by:
\begin{equation*}
{  F(g) = (\un - g) \Tinf_k (g) - \gamma g.}
\end{equation*}
More precisely, we consider~$u=(u_t, t\in \R)$, where~$u_t\in \Delta$
for all~$t\in\R_+$, and  $u$ 
solves in $\cl$:
\begin{equation}
  \label{eq:SIS2}
{   \partial_t u_t = F(u_t)} \quad\text{for } t\in \R_+,
\end{equation}
with initial condition~$u_0\in \Delta$. The value~$u_t(x)=u(t,x)$ models
the  probability  that  an  individual of  feature~$x$  is  infected  at
time~$t$; it  is proved  in~\cite{delmas_infinite-dimensional_2020} that
such a solution~$u$ exists and is unique.

\medskip

An \emph{equilibrium}  of~\eqref{eq:SIS2} is  a function~$g  \in \Delta$
such        that~$F(g)         =        0$.          According        to
\cite{delmas_infinite-dimensional_2020},  there exists  a \emph{maximal}
equilibrium~$\mathfrak{g}$, that is, an  equilibrium such that all other
equilibria~$h\in     \Delta$     are    dominated     by~$\mathfrak{g}$:
$h \leq \mathfrak{g}$.  This maximal equilibrium is obtained as the
long time pointwise limit of  the  SIS  model started  with  its whole  population
infected:     
\( { \lim_{t\rightarrow \infty } u_t=\mathfrak{g}} \) where  
$u_0=\un$.    The   \emph{fraction    of   infected   individuals   at
  equilibrium},~$\I_0$, is thus given by:
\begin{equation*}
\boxed{\I_0=\int_\traits \mathfrak{g}\,  \rd \mu.}
\end{equation*}
For $T$  a bounded  operator on  $\cl$ endowed  with its  usual supremum
norm,  we denote  by~$\norm{T}_{\cl}$ its  operator norm.   The spectral
radius        of        $T$        is        then        given        by
$  \rho(T)= \lim_{n\rightarrow  \infty  } \norm{T^n}_{\cl}^{1/n}$.   The
\emph{reproduction  number}~$R_0$  associated  to the  SIS  model  given
by~\eqref{eq:SIS2}  is  the  spectral   radius  of  the  next-generation
operator:
\begin{equation}
 \boxed{   R_0= \rho (\Tinf_{\kkk}).}
\end{equation}
If~$R_0\leq 1$  (sub-critical and  critical case),  then~$u_t$ converges
pointwise  to~$\zero$  when~$t\to\infty$.  In  particular,  the  maximal
equilibrium~$\mathfrak{g}$   is   equal    to~$\zero$    and
$\I_0=0$.  If~$R_0>1$ (super-critical  case), then~$\zero$  is still  an
equilibrium but  different from the maximal  equilibrium $\mathfrak{g}$,
as~$\I_0=\int_\traits \mathfrak{g} \, \mathrm{d}\mu > 0$.

\subsection{Vaccination strategies}
\label{sec:vacc}

A \emph{vaccination strategy}~$\eta$ of a vaccine with perfect efficiency is an element
of~$\Delta$, where~$\eta(x)$ represents the proportion of \emph{\textbf{non-vaccinated}}
individuals with feature~$x$. Notice that~$\eta\, \mathrm{d} \mu$ corresponds in a sense
to the effective population.
In particular, the ``strategy'' that consists in vaccinating no one
corresponds to $\eta = \un$, the constant function equal to 1, while $\eta
= \zero$, the constant function equal to 0, corresponds to vaccinating everybody.


Recall the definition of the kernel~$f \kk g$ from~\eqref{eq:def-fkg}.
For~$\eta \in \Delta$, the kernel~$\kkk\eta=k\eta/\gamma$ has finite
norm $\norm{\cdot}_{\infty , q}$, so we can consider the bounded
positive operators~$\Tinf_{\kkk \eta }$ and~$\Tinf_{k\eta}$ on~$\cl$.
According to \cite[Section~5.3.]{delmas_infinite-dimensional_2020},
the SIS equation with vaccination strategy~$\eta$ is given by
$u^\eta=(u^\eta_t, t\geq 0)$ solution to~\eqref{eq:SIS2} with~$F$ is
replaced by~$F_\eta$ defined by:
\begin{equation*}
{   F_\eta(g) = (\un - g) \Tinf_{k\eta}(g) - \gamma g.}
\end{equation*}
The quantity~$u_t^\eta(x)=u^\eta(t,x)$ then  represents  the  probability  for  a  non-vaccinated
individual  of feature~$x$  to be  infected at  time $t$;  so at time
$t$ among  the
population  of feature~$x$,  a  fraction $1-\eta(x)$  is  vaccinated,  a
fraction $\eta(x)\, u_t^\eta(x)$ is  not vaccinated and infected, and a
fraction $\eta(x)\, (1-u_t^\eta(x))$ is  not vaccinated and
not infected.

We define the \emph{effective reproduction number} $R_e(\eta)$
associated to the vaccination strategy $\eta$ as  the spectral radius
of the effective next-generation operator~$\Tinf_{\kkk \eta}$:
\begin{equation}\label{eq:def-R_e}
  \boxed{R_e(\eta)=\rho(\Tinf_{\kkk\eta}).}
\end{equation}
For example, for the trivial vaccination strategies we get~$R_e(\un) =
R_0$ and $R_e(\zero) = 0$. 
We also denote by~$\mathfrak{g}_\eta$ the corresponding maximal
equilibrium and, using that $\eta \, \rd \mu$ is the effective
population, we define   the \emph{effective fraction
  of infected individuals at equilibrium} as:
\begin{equation}
  \label{eq:def-I}
\boxed{\I(\eta)=\int_\traits \mathfrak{g_\eta}\, \eta\,  \rd \mu.}
\end{equation}
For example, we have  $\I(\un)=\I_0$ and $\I(\eta)=0$ for all $\eta\in
\Delta$ such that $R_e(\eta)\leq  1$. 

\subsection{Optimal strategies}

For a vaccination strategy $\eta\in \Delta$, we consider its loss
$\loss(\eta)$,  given either by the effective
reproduction number  ($\loss=R_e$) or by the effective fraction
of infected individuals at equilibrium ($\loss=\I$). 
Following~\cite{ddz-theory-optim}, we  measure the cost for  the society
of a vaccination strategy (production,  diffusion, ...) by a nonnegative
function~$C$ defined  on $\Delta$.  We shall  concentrate on  the affine
case:
\begin{equation*}
\boxed{  C(\eta) = \int_\traits  (\un-\eta)\, \costad \, \rd \mu}
\end{equation*}  
where the nonnegative function $\costad \in L^1$ represents the
feature-dependent
cost of vaccinating individuals.   Notice that
doing nothing costs nothing, that is,~$C(\un)=0$.  A simple and
natural choice is the uniform cost~$\costu$ corresponding to
$\costad=\un$.

\medskip

Let us note  that if $H$ is any of the three functionals $R_e$, $\I$ or $C$,
and if $\eta_1= \eta_2$ $\mu$-a.s., then $H(\eta_1) = H(\eta_2)$. 
Following~\cite{ddz-theory-optim} we therefore consider
the set of vaccination strategies as a subset of $L^\infty $:
\begin{equation}
   \label{eq:def-D}
  \boxed { \Delta=\{\eta\in L^\infty \, \colon\, 0      \leq  \eta\leq
 1 \quad \mu-\text{a.s.}\}.}
\end{equation}

\medskip

In  \cite[Section~4]{ddz-theory-optim},  we  formalized  and  study  the
problem of  optimal allocation strategies  for a perfect vaccine  in the
SIS model.  This  question may be viewed as  a bi-objective minimization
problem,  where one  tries to  minimize simultaneously  the cost  of the
vaccination and  its corresponding loss:
\begin{equation*}
  {\min_{\Delta} (C,\loss).}
\end{equation*}


We   call   a   strategy~$\eta_\star$
\emph{Pareto optimal} if no other strategy is strictly better:
\[
  C(\eta)< C(\eta_\star) \implies \loss(\eta) > \loss(\eta_\star)
  \quad\text{and}\quad
  \loss(\eta)< \loss(\eta_\star) \implies C(\eta) > C(\eta_\star).
\]
The set of Pareto optimal strategies will be denoted by~$\cP\subset \Delta$, and we define the
\emph{Pareto frontier} as the set of Pareto optimal outcomes:
\[
\boxed{\F =
  \{ (C(\eta_\star),\loss(\eta_\star)) \, \colon \, \eta_\star \in
  \cP \}.}
\]
We call a strategy
$\eta^\star$  \emph{anti-Pareto optimal} if no other strategy is strictly
worse, that is, $  C(\eta)> C(\eta_\star) \implies \loss(\eta) <
\loss(\eta_\star)$ and $
  \loss(\eta)> \loss(\eta_\star) \implies C(\eta) < C(\eta_\star)$.
The set of anti-Pareto optimal strategies will be denoted by~$\cpa\subset \Delta$, and we define the
\emph{anti-Pareto frontier} as the set of anti-Pareto optimal outcomes $\AF =
  \{ (C(\eta ^\star),\loss(\eta ^\star)) \, \colon \, \eta \star \in
  \cpa\}$.
  We refer to~\cite{ddz-theory-optim} for an extensive study and alternate
  characterizations of 
  the Pareto and anti-Pareto frontiers; let us simply mention that
  under our assumptions both frontiers are non-trivial.

\subsection{Parameters of the SIS model  in a nutshell}
\label{sec:hyp}
Let us summarize the setup. The SIS model is given by a probability space $(\traits,
\cf, \mu)$, a positive recovery rate function $\gamma\in \cl (\traits, \cf)$,
a transmission rate kernel $k$ (that is,  a measurable
nonnegative function defined on $\traits^2$) such that
$\norm{k/\gamma}_{\infty ,q}<\infty $ for some $q\in (1, +\infty )$,
see~\eqref{eq:bded}, and an affine cost function with a nonnegative  density
$\costad\in L^1(\traits, \cf, \mu)$. 
We denote the  parameters of the SIS model by:
\[
\param=[(\traits, \cf,  \mu), (k, \gamma), \costad\, ].
\]
Finally, we write $H[\param]$ to emphasize the dependence of any quantity
$H$ on the parameters: for example $R_e[\param](\eta)$ is the effective
reproduction number associated to the vaccination strategy $\eta$ in the model
defined by $\param$.

\section{Equivalence of models by coupling}\label{sec:equivalent}

We now define our main tool: the \emph{coupling} of two SIS models, which gives
rise to a notion of \emph{conjugation} between functions defined on the first
and the second model. This tool is then used to state our main results. All
proofs are postponed to Section~~\ref{sec:proof-coupling}.

\begin{remark}[Graphons and weak isometry]
  In Section~\ref{sec:exple-couple},
we present an example where discrete models can be represented as a
continuous models and an example based on measure preserving
transformation in the spirit of the graphon theory.  We refer the
reader to \cite{janson2010graphons} for similar developments in the
graphon setting.
\end{remark}

 \subsection{On measurability}
 \label{sec:measurable}
 
 Let us recall some well-known facts on measurability.  Let~$(E, \ce)$
 and~$(E', \ce')$ be two measurable spaces.  If~$E'=\R$, then we
 take~$\ce'=\mathcal{B}(\R)$ the Borel~$\sigma$-field.  Let~$f$ be a
 function from~$E$ to~$E'$.  We denote by
 $\sigma(f)=\{f^{-1}(A)\, \colon \, A\in \ce'\}$ the~$\sigma$-field
 generated by~$f$.  In particular the function~$f$ is measurable
 from~$(E, \ce)$ to $(E', \ce')$ if and only
 if~$\sigma(f)\subset \ce$.  Let~$\varphi$ be a measurable function
 from~$(E, \ce)$ to~$(E', \ce')$.  For~$\nu$ a measure on~$(E, \ce)$,
 we write~$\nu'=\varphi_\# \nu$ for the push-forward measure
 on~$(E', \ce')$ of the measure~$\nu$ by the function~$\varphi$; by
 definition of~$\nu'$, for a nonnegative measurable function~$g$
 defined from~$(E', \ce')$ to~$(\R, \mathcal{B}(\R))$, we have:
\begin{equation}\label{eq:push}
  \int_{E'} g \, \mathrm{d} \nu'= \int _{E} g\circ
  \varphi \, \mathrm{d} \nu.
\end{equation}
In particular, if $f,g$ are measurable functions defined
from~$(E', \ce')$ to some measurable space, then we have that
$\nu'$-a.e.\ $f=g$ if and only if $\nu$-a.e.\
$ g\circ \varphi = f\circ \varphi $.  Thus, if $g$ belongs to
$ L^p(E', \ce', \nu')$, then $g\circ \varphi$ is well defined as an
element of $L^p(E, \ce, \nu)$.

Let~$f$ be a measurable function from~$(E, \ce)$ to~$(\R,
\mathcal{B}(\R))$. We recall (see for example~\cite[Lemma 1.14]{Kal21}) that:
\begin{equation}
 \label{eq:def-f-phi0}
 \sigma(f)\subset \sigma(\varphi)\,
 \Longrightarrow\, f=g\circ \varphi,
\end{equation}
for some measurable function $g$ from~$(E', \ce')$ to~$(\R, \mathcal{B}(\R))$.

In what follows the random variables are defined on some probability
space $(\Omega_0, \cf_0, \P)$. 

\subsection{Coupling and conjugate functions}
Let $(E_1,\ce_1,\mu_1)$ and $(E_2,\ce_2,\mu_2)$ be measurable spaces.
A \emph{coupling} is a measure $\pi$ on $(E_1\times E_2, \ce_1 \otimes \ce_2)$ with
marginals $\mu_1$ and $\mu_2$. By abuse of notation we also call coupling
a random variable $Z=(Z_1,Z_2)$ with distribution~$\pi$, and also say
that $E_1$ and $E_2$ are coupled trough $Z$.

We introduce a notion of conjugacy whose basic properties are similar to
convex conjugation.

\begin{definition}[Conjugate functions]
  Let  $(E_1,\ce_1,\mu_1)$ and $(E_2,\ce_2,\mu_2)$ be coupled through $
  (Z_1,Z_2)$. 
  Let $f_i\in L^1(E_i,\mu_i)$ for $i=1, 2$. 
  The \emph{conjugate} $f_1^*$ of $f_1$ is the element of $L^1(E_2)$ defined by:
  \[
    f_1^*(Z_2) = \esp{ f_1(Z_1) \given \common}
    \quad\text{with}\quad
\common = \sigma(Z_1)\cap\sigma(Z_2); 
  \]
  its existence is justified by~\eqref{eq:def-f-phi0}. 
  Similarly $f^*_2\in L^1(E_1)$ is defined by $f_2^*(Z_1) =
  \esp{f_2(Z_2)\given \common}$.

  The pair $(f_1,f_2)$ is
  called \emph{conjugate} if   $f_1 = f_2^*$ and $f_2 = f_1^*$; it is
  called \emph{pre-conjugate} if the  pair $(f_2^{*},f_1^*)$ is
  conjugate (that is, $ f_2^*=f_1^{**}$ and  $f_1^*=f_2^{**}$). 
\end{definition}

Notice that a conjugate pair is  also pre-conjugate, but the converse is
false in general. 

We shall see below that if the transmission kernels,
recovery functions and the density of the cost functions of two SIS
model are conjugate, then any vaccinations strategies which are
pre-conjugate have the same loss and cost, and thus are (anti-)Pareto
optima simultaneously.

We first give another characterization of the conjugation.

\begin{lemma}[Characterization of conjugation]
  \label{lem:conj0}
 Let  $(E_1,\ce_1,\mu_1)$ and $(E_2,\ce_2,\mu_2)$ be coupled through $
 (Z_1,Z_2)$. 
 Let $f_i\in L^1(E_i)$ for $i=1, 2$. We have:
       \[
         (f_1,f_2) \quad\text{is conjugate}\quad
         \Longleftrightarrow \quad f_1(Z_1) = f_2(Z_2) \quad \pi\text{-a.s..}
       \]
       If  the  pair  $(f_1,f_2)$   is  conjugate,  then  $f_i(Z_i)$  is
       $\common$-measurable          for          $i=1,2$,          with
       $\common = \sigma(Z_1)\cap\sigma(Z_2)$.
\end{lemma}
\begin{proof}
The proof is immediate as,  for  $X$ and $Y$ integrable random
  variables and a sub-$\sigma$-field $\common$,  the equalities $\E[X|\common]=Y$ and $\E[Y|\common]=X$ imply that  a.s.\ $X=Y$. 
\end{proof}

We shall complete the next result with other properties in
Section~\ref{sec:proof-coupling}.

\begin{lemma}[Properties of conjugation]
  \label{lem:conj1}
 Let  $(E_1,\ce_1,\mu_1)$ and $(E_2,\ce_2,\mu_2)$ be coupled through $Z =
 (Z_1,Z_2)$. Let $f\in L^1(E_1)$.

  \begin{enumerate}[(i)]
    \item \label{it:ff*}
       The pair $(f,f^*)$ is  pre-conjugate and 
      the pair  $(f^{**}, f^*)$ is  conjugate.

    \item  \label{it:ff**}
      Set $\common = \sigma(Z_1)\cap\sigma(Z_2)$. We have:
    \[
f(Z_1)  \text{ is $\common$-measurable} 
\quad\Longleftrightarrow \quad
f=f^{**}
 \quad\Longleftrightarrow \quad
  (f,f^*)\text{ is conjugate.}
\]
\end{enumerate}
\end{lemma}

\begin{proof}
    By definition, we have
  \(
  f^{**}(Z_1) = \esp{ f^*(Z_2) \given \common} = \esp{ \esp{f(Z_1) \given\common} \given \common},\)
  which yields that $ f^{**}(Z_1)=f^*(Z_2)$. 
By Lemma~\ref{lem:conj0} this implies that $(f^{**}, f^*)$ is conjugate
and thus that $(f, f^*)$ is pre-conjugate. This gives~\ref{it:ff*}. 

We now prove~\ref{it:ff**}.
Notice first that $f^{**}=f$ is equivalent to the pair  $(f, f ^*)$ being conjugate. 
Secondly, if $(f, f ^*)$ is conjugate, then by Lemma~\ref{lem:conj0}, we
get  that 
$f(Z_1)$ is
$\common$-measurable. Conversely,  if $f(Z_1)$ is
$\common$-measurable, we deduce  that $f^*(Z_2)=f(Z_1)$ and thus
$f^{**}=f$. 
\end{proof}

Let two  spaces $E_1$ and  $E_2$ be  coupled through $\pi$.  The product
spaces $\bE_1 = E_1\times E_1$ and $\bE_2 = E_2\times E_2$ may always be
coupled          through           the          random          variable
$(\bZ_1,\bZ_2)   =  ((X_1,Y_1),(X_2,Y_2))$,   where   the  two   vectors
$(X_1,X_2)$ and $(Y_1,Y_2)$ are  independent and follow the distribution
$\pi$. We denote the distribution  of $(\bZ_1,\bZ_2)$ by $\extended$ and
call  it  the  \emph{extended  coupling}. Conjugates  are  preserved  by
extension in the following sense; the proof is given in
Section~\ref{sec:proof-extended}. 
\begin{lemma}[Extended coupling and conjugacy]
  \label{lem:extended}
  If the measurable function $g:\bE_1\to \mathbb{R}$ only
  depends on its first argument, $g(x_1,y_1)= f(x_1)$, then $g^*(X_2,Y_2) = f^*(X_2)$
  (where $g^*$ is the conjugate through $\extended$ and $f^*$ the conjugate through $\pi$).
\end{lemma}

\subsection{Coupled models}

We consider the SIS models
$\param_i=[(\traits_i, \cf_i, \mu_i), (k_i, \gamma_i), \costad_{i}]$
for $i=1,2$.  In what follows, we simply write $\Delta_i$ for
the set of functions $\Delta$, see~\eqref{eq:def-D}, in the model
$\param_i$.

\begin{theorem}[Coupling, equilibria and optimal vaccinations]
  \label{thr:main}
  Consider two SIS models  $\param_1$ and $\param_2$, with
  a coupling
 between
  $(\traits_1, \cf_1, \mu_1)$ and $(\traits_2, \cf_2, \mu_2)$.
 Let $\eta_i\in \Delta_i$ be a vaccination strategies for the SIS model
$i=1,2$. 

\begin{enumerate}[(i)]
    
\item\label{it:th-Re}
  If the pair $(k_1/\gamma_1,k_2/\gamma_2)$ is conjugate (for the
  extended coupling), then 
  \[
    (\eta_1, \eta_2)
    \quad\text{is pre-conjugate}\quad
     \Longrightarrow\quad 
     R_e[\param_1](\eta_1) = R_e[\param_2](\eta_2).
   \]
   

  \item\label{it:th-I}
   If  both pairs $(k_1,k_2)$ and $(\gamma_1,\gamma_2)$ are
   conjugate,
then the equilibria are (pre-)conjugate: if $g_1$ is an equilibrium of
    $\param_1$, then there exists an  equilibrium $g_2$ of $\param_2$
    such that the pair $(g_1, g_2)$ is conjugate.  
    We also have:
  \[
    (\eta_1, \eta_2)
    \quad\text{is pre-conjugate}\quad
     \Longrightarrow\quad 
    \I[\param_1](\eta_1) = \I [\param_2](\eta_2).
   \]


 \item \label{it:main}
     Suppose the assumptions of item~\ref{it:th-Re}, for $\loss=R_e$, or of
  item~\ref{it:th-I}, for $\loss=\I$,  hold.  Assume also that
the pair  $(\costad_1,\costad_2)$ is conjugate.   If  the pair
$(\eta_1, \eta_2)$ is pre-conjugate, then:: 
     \begin{equation}
   \label{eq:couple-PO}
  \text{$\eta_1$ is (anti-)Pareto optimal for
    $\param_1$}
  \, \Longleftrightarrow\,
  \text{$\eta_2$ is (anti-)Pareto optimal for
      $\param_2$}.
\end{equation}
For $\eta\in \Delta_1$, we have $\eta^*\in \Delta_2$ and
$H[\param_1](\eta) = H[\param_2](\eta^*)$ for $H$ equal to the loss
$\loss$ or the cost $C$; in particular, if $\eta$ is (anti-)Pareto
optimal for $\param_1$, then its conjugate $\eta^*$ is (anti-)Pareto
optimal for $\param_2$.
\end{enumerate}
\end{theorem}

  As a
direct consequence, we get the following result, where  the set of
outcomes is defined as $\FF=\{(C(\eta), \loss(\eta)), \eta\in\Delta\}$. 
\begin{corollary}[Coupling and frontiers]
  \label{cor:=frontier}
  Let  $\param_1$ and  $\param_2$ be coupled  SIS models, with conjugate
  parameters $\gamma$, $\costad$ and $k$. For any of the two choices $\loss\in  \{R_e, \I\}$,
  the models $\param_1$ and $\param_2$ have the same set of outcomes $\FF$
  and the same (anti-)Pareto frontiers $\F$ and $\AF$. 
\end{corollary}

\begin{remark}[Obvious couplings]
  \label{rem:couplage}
    If  the costs  are uniform  in both  models
    $\param_1$  and   $\param_2$,  then the pair $(\costad_1,\costad_2)$ is
    trivially conjugate as both functions are a.s.\ constant equal to $1$.

    Using a trivial coupling, one sees that
the recovery rate and transmission kernel in
    the SIS model could have been defined only almost everywhere without
    affecting the set of outcomes and the (anti-)Pareto frontiers. 
  \end{remark}

  The coupling hypotheses are strong and give strong results, allowing
  to compare equilibria and vaccinations between models. Let us note that
  other, weaker ways of comparing models exist, and may yield interesting results.
  \begin{remark}[Life without coupling --- normalizing $\gamma$ and $\costad$]
    If we are only interested in the loss function $\loss=R_e$, various invariance properties
    of the spectral radius  may be used to
    normalize models. Indeed, consider a SIS model $\param=[(\traits,  \cf, \mu), (k, \gamma), \costad]$
    for which  both $\gamma$ and $\costad$ are bounded away from zero, and assume
    without loss of generality that $\int_\traits \costad\, \mathrm{d}\mu = 1$. 
   Define another model by 
   $\param_0=[(\traits, \cf, \mu_0), (k_0, \gamma_0), \costad_0]$, where:
   \begin{align*}
     \mu_0(\mathrm{d} x) &= \costad(x) \, \mu(\mathrm{d} x),
     &
       k_0 &= k/(\costad\gamma),
     &
       \gamma_0  &= \costad_0 = \un.
   \end{align*}
Notice that as~\eqref{eq:bded} holds
for the model $\param$, then it also holds for the model $\param_0$ as
we assumed $\costad$ to be bounded away from 0.

We trivially have
$\Delta(\param)=\Delta(\param_0)$.  Clearly, we have $C(\eta)$ for the
model $\param$ is equal to $\costu(\eta)$ for the model $\param_0$.
Using also that $L^p(\mu)$ and $L^p(\mu_0)$ are compatible
(see~\cite[Lemma~2.2]{ddz-Re}) and the corresponding integral
operators are consistent (see~\cite[Section~2.2]{ddz-Re} and
Lemma~2.1(iii)), we get that $ R_e (\eta)$ for the model $\param$ is
equal to $ R_e (\eta)$ for the model $\param_0$, for all strategies
$\eta\in \Delta$. In particular the (anti-)Pareto optimal strategies
and the (anti-)Pareto frontiers are the same for the two models.
Therefore 
we may focus on $\param_0$ and assume without loss of generality that the only dependence on the features
is in the transmission kernel, while  both the vaccination cost and
the 
recovery rate are uniform.
\end{remark}

\section{Examples of couplings}
\label{sec:exple-couple}

We discuss three examples, all of which are built on the following special case
of coupling, each one taking a slightly different point of view.

\begin{lemma}[Deterministic coupling]
  \label{lem:deterministic_coupling}
  Let  $(E_1,\ce_1,\mu_1)$ and $(E_2,\ce_2,\mu_2)$ be two probability spaces
  and assume that $\phi:E_1\to E_2$ is measurable and pushes $\mu_1$ forward to $\mu_2$.
  Then $E_1$ and $E_2$ are coupled through
  $(X_1,\phi(X_1))$, with $X_1 \sim \mu_1$,
and  for any
  two functions $f_i\in L^1(E_i)$, $i=1,2$ we have, with $\E_1$ the
  expectation w.r.t.\  $\mu_1$: 
  \begin{enumerate}
    \item $f_2^*=f_2 \circ \phi$, $f_1^* \circ \phi=\E_1[f_1 \, |\, 
        \sigma(\phi)]$ and $f_2^{**} = f_2$;
  \item The pair $(f_1,f_2)$ is  conjugate if and only if $f_1 = f_2\circ \phi$;
  \item The pair $(f_1,f_2)$ is  pre-conjugate if and only if
    $f_2 = f_1^*$;
      \item The pair of kernels $(k_1, k_2)$ (respectively on $\bE_1$
        and $\bE_2$) is  conjugate
        (through the extended coupling) 
        if and only if $\mu_1(\rd x_1)\otimes\mu_1(\rd y_1)$-a.e.\  $k_1(x_1,y_1) = k_2(\phi(x_1), \phi(y_1))$.
        
  \end{enumerate}
\end{lemma}
The  proof is elementary and left to the reader. 

\subsection{Starting from $E_1$: model reduction using deterministic coupling}
\label{sec:exple-esp-cond}
We consider a SIS model
$\param_1=[(\traits_1,\cf_1,\mu_1),(k_1,\gamma_1), \costad_{1}]$. Let
$\phi$ be a measurable function from $(\traits_1, \cf_1)$ to
$(\traits_2, \cf_2)$, let $\mu_2$ be the push-forward $\phi_\#\mu_1$,
and consider the coupling given by $(X_1,\phi(X_1))$ where
$X_1\sim \mu_1$.  By Lemma~\ref{lem:deterministic_coupling}, the
functions $\costad_1$, $\gamma_1$ and $k_1$ will be part of conjugate
pairs for this coupling if and only if they all factor through $\phi$,
in the sense that for some functions $\costad_2$, $\gamma_2$ on
$\traits_2$ and $k_2$ on $\traits_2\times \traits_2$:
\begin{equation}
  \label{eq:subset-tribu}
  \costad_1  = \costad_2\circ\phi, \quad
  \gamma_1  = \gamma_2\circ\phi
  \quad\text{and}\quad
   k_1(\cdot, \cdot) = k_2(\phi(\cdot), \phi(\cdot)).
\end{equation}

  If that is the case, then by Theorem~\ref{thr:main} and
  Lemma~\ref{lem:deterministic_coupling}, the vaccination strategy 
  $\eta_1\in \Delta_1$ is (anti-)Pareto optimal for $\param_1$
  if and only if its conjugate   $\eta_1^*$ defined by
    $\eta_1^* \circ \phi =
    \esp{\eta_1\given \sigma(\phi)}$  is 
  (anti-)Pareto optimal for the
  simplified model $\param_2$.
  In words, the behaviour of an individual $x$ only depends on $\phi(x)$,
  and in the trait space $\traits_2$, individuals with identical behavior
  are merged.

We may deduce the following result.
\begin{corollary}[Model reduction]
  \label{cor:couplage}
  Let $\param=[(\traits,\cf,\mu),(k,\gamma), \costad]$ be a SIS model
  with loss function $\loss\in \{R_e, \I\}$. Let $\cg\subset\cf$ be a
  $\sigma$-field such that $\gamma$ and $\costad$ are $\cg$-measurable
  and $k$ is $\cg\otimes\cg$-measurable. Then, for any  $\eta\in \Delta$, we have, with $\E_\mu$ the expectation
  w.r.t.\  $\mu$:
  \begin{equation}
    \label{eq:P0-h-Eh}
    \text{$\eta$ is (anti-)Pareto optimal} \quad \Longleftrightarrow\quad
    \text{$\E_\mu[\eta\, |\, \cg]$ is (anti-)Pareto optimal}.
  \end{equation}
\end{corollary}

\begin{proof}
  Denote with a subscript $1$ the parameters of the original model
  (e.g., set $\traits_1 = \traits$).  Let
  $\param_2 = [(\traits_2,\cf_2,\mu_2),(k_2,\gamma_2),\costad_2]$
  where most parameters are the same: $\traits_2 = \traits$,
  $k_2 = k$, $\gamma_2 = \gamma$, $\costad_2 = \costad$, but we equip
  $\traits_2$ with $\cf_2 = \cg$, and the measure
  $\mu_2 = (\mu_1)_{|\cg}$.  Note that this is legitimate, in the sense
  that the measurability hypotheses on $(k,\gamma,\costad)$, imply
  that $\gamma_2,\costad_2$ are measurable from $(\traits_2,\cf_2)$ to
  $(\dR,\cb(\dR))$ and $k_2$ is measurable on the product space
  $(\traits_2\times \traits_2, \cF_2\otimes \cf_2)$.

  Now we define $\phi:\traits_1\to \traits_2$ by $\phi(x) = x$; since
  $\cg\subset \cf$, $\phi$ is measurable from
  $(\traits_1,\cf_1) = (\traits,\cf)$ to
  $(\traits_2,\cf_2) = (\traits,\cg)$. This function defines a deterministic coupling
  between the two spaces. Since $\phi$ is the identity if we forget the
  measure structure, it is clear that $\gamma_1 = \gamma_2\circ \phi$,
  $\costad_1 = \costad_2\circ \phi$ and $k_1(\cdot,\cdot) = k_2(\phi(\cdot),\phi(\cdot))$,
  so that all three pairs of functions are conjugate, by Lemma~\ref{lem:deterministic_coupling}.
  Applying Theorem~\ref{thr:main} twice, we get that $\eta\in \Delta$ is Pareto-optimal for
  $\param_1$ if and only if $\eta^*$ is Pareto-optimal for $\param_2$, if
  and only if $\eta^{**}$ is Pareto-optimal  for
  $\param_1$.

  Let us finally  identify $\eta^{**}$. Let $X$ be $\mu_1$-distributed.
  The coupling is $(Z_1,Z_2)$ where $Z_1 = X$ and $Z_2 = \phi(Z_1)$, so
  $\common = \sigma(Z_1)\cap \sigma(Z_2) = \sigma(Z_2) = X^{-1}(\phi^{-1}(\cg)) = X^{-1}(\cg)$.
  By definition we have
  \( \eta^{**}(X) = \esp{\eta(X)\given X^{-1}(\cg)}. \)
We deduce that  $\eta^{**} =\E_\mu[\eta\, |\, \cg]$ as for any $B\in \cg$
and $A = X^{-1}(B)$: 
  \begin{align*}
    \esp{\eta(X) \ind{A}} = \esp{\eta(X) \ind{B}(X)} = \int_\traits
    \eta(x) \ind{B}(x) \, \mu_1(\rd x)
    &= \int_\traits \E_\mu[\eta\, |\, \cg](x)\,  \ind{B}(x) \, \mu_1(\rd x) \\
    &= \esp{ \E_\mu[\eta\, |\, \cg](X) \ind{A}}.\qedhere
    \end{align*}
  \end{proof}

\subsection{Linking $E_1$ and $E_2$: discrete and continuous models}
\label{sec:dis-cont}

We now consider a particular case, and formalize how finite population models can be seen as
images of models with a continuous population.
We denote by $\cb([0,1))$ and by $\mathrm{Leb}$ the Borel~$\sigma$-field
and the Lebesgue measure on $[0,1)$.

Let
$\traits_\mathrm{d}\subset \N$,~$\cf_\mathrm{d}$ the set of subsets
of~$\traits_\mathrm{d}$ and~$\mu_\mathrm{d}$ a probability measure
on~$\traits_\mathrm{d}$. Without loss of generality, we can assume
that $\mu_\mathrm{d}(\{\ell\})>0$ for all
$\ell\in \traits_\mathrm{d}$. We set~$\traits_\mathrm{c}=[0, 1)$,
$\cf_\mathrm{c}=\cb([0,1))$ and let~$\mu_\mathrm{c}$ be a probability
measure on $(\traits_\mathrm{c},\cf_\mathrm{c})$ without atoms (for
example one can take the Lebesgue measure $\mathrm{Leb}$). Let
$(B_\ell, \ell\in\traits_\mathrm{d})$ be a partition of $[0,1)$ in
measurable sets such that
$\mu_\mathrm{c}(B_\ell) = \mu_\mathrm{d}(\{\ell\})$ for all
$\ell\in \traits_\mathrm{d}$.
The map $\phi:\traits_{\mathrm{c}} \to \traits_{\mathrm{d}}$ defined
by $\phi(x) = \sum \ell \ind{B_\ell}(x)$ clearly defines a deterministic
coupling between $\mu_{\mathrm{c}}$ and $\mu_{\mathrm{d}}$. 
 If the
kernels $k_\mathrm{d}$ on $\traits_\mathrm{d}$ and $k_\mathrm{c}$ on
$\traits_\mathrm{c}$ and the functions
$(\gamma_\mathrm{d}, \costad_{\mathrm{d}})$ and
$(\gamma_\mathrm{c}, \costad_{\mathrm{c}})$ are related through the
formula:
\[
  \gamma_\mathrm{c}(x)=\gamma_\mathrm{d}(\ell), \quad
  \costad_{\mathrm{c}}(x)=\costad_{\mathrm{d}}(\ell) 
  \quad\text{and}\quad
  k_\mathrm{c}(x,y) = k_\mathrm{d}(\ell,j) \quad \text{for $x\in
    B_\ell$, $y \in B_j$ and $\ell, j\in \traits_\mathrm{d}$},
\]
then all pairs are conjugate, and all the hypotheses of Theorem~\ref{thr:main}
an Corollary~\ref{cor:=frontier}
are satisfied. 

Roughly speaking, we can blow up
the atomic part of the measure~$\mu_\mathrm{d}$ into a continuous
part, or, conversely, merge all points that behave similarly for
$k_\mathrm{c}$, $\gamma_\mathrm{c}$ and $ \costad_{\mathrm{c}}$ into
an atom, without altering the Pareto frontier.

\begin{example}[The stochastic block model]
  \label{ex:sbm2}
  To be more concrete, we consider the  so called stochastic block model, with
  2  populations for simplicity and  give in  this
  elementary case the corresponding  discrete and  continuous models. Then, we explicit
  the  relation  with the  formalism  of  the  same model  developed  in
  \cite{lajmanovich1976deterministic}   by    Lajmanovich   and
  Yorke.  For simplicity, we assume the cost is
    uniform (that is, $\costad=\un$), so that the
conjugation condition for the costs is trivially satisfied.
  \medskip

  The discrete SIS model is defined on $\traits_\mathrm{d}=\{1,2\}$
  with the probability measure $\mu_\mathrm{d}$ defined by
  $\mu_\mathrm{d}(\{1\})=1-\mu_\mathrm{d}(\{2\})=p$ with $p\in (0,1)$,
  and a transmission kernel~$k_\mathrm{d}$ and recovery
  function~$\gamma_\mathrm{d}$ given by the matrix and the vector:
  \[
    k_\mathrm{d}=\begin{pmatrix} k_{11} & k_{12}\\
    k_{21} & k_{22}
  \end{pmatrix}
  \quad\text{and}\quad
  \gamma_\mathrm{d}=\begin{pmatrix} \gamma_1\\ \gamma_2 \end{pmatrix}.
  \]
  Notice $p$ is the relative size of population 1.  The corresponding
  discrete model is
  $\param_\mathrm{d} =
  [(\{1,2\},\cf_\mathrm{d},\mu_\mathrm{d}),(k_\mathrm{d},\gamma_\mathrm{d}),\costad_{\mathrm{d}}=\un
  ]$; see Figure~\ref{fig:discrete-model}.
 
  \medskip

  The continuous SIS model is defined on the probability
  space~$(\traits_\mathrm{c}=[0, 1), \cf_\mathrm{c}=\cb([0,
  1)),\mu_\mathrm{c}=\mathrm{Leb})$.  The
  segment~$\traits_\mathrm{c}=[0,1)$ is partitioned into two
  intervals~$B_1=[0,p)$ and~$B_2=[p, 1)$, the transmission
  kernel~$k_\mathrm{c}$ and recovery rate $\gamma_\mathrm{c}$ are
  given by:
  \[
    k_\mathrm{c}(x,y) = k_{ij}
    \quad\text{and}\quad
    \gamma_\mathrm{c}(x)= \gamma_i
    \quad \text{for $x\in B_i$, $y\in B_j$, and  $i,j\in \{1, 2\}$}.
  \]
  The corresponding continuous model is
  $\param_\mathrm{c} = [(\traits_\mathrm{c},
  \cf_\mathrm{c},\mu_\mathrm{c}),
  (k_\mathrm{c},\gamma_\mathrm{c}),\costad_{\mathrm{c}} =\un ]$; see
  Figure~\ref{fig:continuous-model}.  By the general discussion above,
  these two models have the same (anti-)Pareto frontiers, and their
  equilibria and optimal vaccinations may be transferred to one another by
  conjugation. Let us note that, in this  example, by Lemma~\ref{lem:deterministic_coupling}
  a function $f_{\mathrm{d}}$ on $\traits_{\mathrm{d}} = \{1,2\}$ and
  $f_{\mathrm{c}}$ on $\traits_{\mathrm{c}}$ are~:
  \begin{itemize}
  \item pre-conjugate if and only if
    $\frac{1}{\mu_\mathrm{c}(B_i)} \int_{B_i} f_{\mathrm{c}}\, \rd \mu_
    \mathrm{c}  = f_{\mathrm{d}}(i)$, for $i=1,2$; 
  \item conjugate if and only if
    $ f_{\mathrm{c}}(x) = f_{\mathrm{d}}(i)$, a.e.\ for $x\in B_i$ and
    $i=1, 2$.
  \end{itemize}  
  Therefore,  in this  case, the  optimal strategies  of the  continuous
  model are easily  deduced from the optimal strategies  of the discrete
  model.

  To conclude this example,  using the formalism of the discrete model
$\param_\mathrm{d}$,  the next-generation matrix~$K$ in the setting of
  \cite{lajmanovich1976deterministic}, and the effective next-generation matrix~$K_e(\eta)$ when the
  vaccination strategy~$\eta$ is in force (recall~$\eta_i$ is the proportion of population
  with feature~$i$ which is not vaccinated), are given by:
  \[
    K=
    \begin{pmatrix}
      \kkk_{11} \,p\, & \kkk_{12} \,(1-p) \\
      \kkk_{21}\, p\, & \kkk_{22} \,(1-p)
    \end{pmatrix}
    \quad\text{and}\quad
    K_e(\eta)
    = \begin{pmatrix}
      \kkk_{11} \,p\, \eta_1 & \kkk_{12} \,(1-p)\, \eta_2 \\
      \kkk_{21}\, p\, \eta_1 & \kkk_{22} \,(1-p)\, \eta_2
      \end{pmatrix}
      \quad\text{with}\quad
      \kkk_{ij}=k_{ij}/ \gamma_j.
  \]
\end{example}

\begin{figure}
 \savebox{\largestimage}{\includestandalone[page=1]{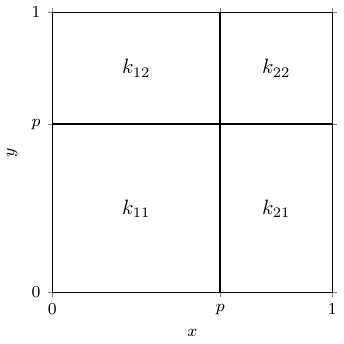}}
 \begin{subfigure}[b]{0.48\textwidth}
 \centering
 \raisebox{\dimexpr.5\ht\largestimage-.5\height}{%
 \includestandalone[page=2]{fig}}
 \caption{Discrete model: kernel~$k_\mathrm{d}$ on
 $\traits_\mathrm{d}=\{1, 2\}$ with the measure
 $\mu_\mathrm{d}=p\delta_1+ (1-p)\delta_2$.}
 \label{fig:discrete-model}
 \end{subfigure}
 \begin{subfigure}[b]{0.48\textwidth} \centering
 \usebox{\largestimage}
 \caption{Continuous model: kernel~$k_\mathrm{c}$ on~$\traits_\mathrm{c}=[0,1)$
 with the Lebesgue measure $\mu_\mathrm{c}$.}
 \label{fig:continuous-model}
 \end{subfigure}
 \quad
 \caption{Coupled discrete
   model (left) and continuous model (right).}
 \label{fig:discrete-and-continuous}
\end{figure}

\subsection{
    Starting from \texorpdfstring{$E_2$}{E2}: measure preserving function}
\label{sec:preserver}
Finally, let us briefly discuss an example motivated by the theory of
graphons, which are indistinguishable by measure preserving
transformation, see \cite[Sections~7.3 and~10.7]{lovasz_large_2012}.

Let~$(\traits, \cf, \mu)$ be a probability space.  We say a measurable
function~$\varphi: \traits \rightarrow \traits$  is \emph{measure
  preserving} if~$\mu= \varphi _\# \mu$. For example the
function~$\varphi:x\mapsto 2x \mod(1)$ defined on the probability
space~$([0, 1], \cb([0, 1], \mathrm{Leb})$ is measure preserving. Note
it is not one-to-one in general.

Now  consider a  SIS model  with parameters  $\param_2 =
[(\traits,\cf,\mu),(k,\gamma),\costad]$  and a
measure         preserving        function         $\phi$.        Define
$\gamma_1  =  \gamma\circ  \phi$,  $\costad_1  =  \costad\circ\phi$  and
$k_1(\cdot,\cdot)     =     k_2(\phi(\cdot),     \phi(\cdot))$.     Then
the models $\param_1     =    [(\traits,\cf,\mu),(k_1,\gamma_1),\costad_1]$     and
$\param_2 $  are coupled and all
consequences of Theorem~\ref{thr:main} and Corollary~\ref{cor:=frontier}
hold.  Roughly speaking, we can give different labels to the features of
the population without altering the (anti-)Pareto frontiers.

\section{Proofs}
\label{sec:proof-coupling}
\subsection{Elementary properties of conjugation}


We give further technical properties of the conjugation.

\begin{lemma}[Other properties of conjugation]
  \label{lem:conj}
Let  $(E_1,\ce_1,\mu_1)$ and $(E_2,\ce_2,\mu_2)$ be coupled through $Z =
(Z_1,Z_2)$.  Let $f\in L^1(E_1)$ and $f_i\in L^1(E_i)$, $i=1,2$.  

  \begin{enumerate}[(i)]
    \item  We have      $f^{***} = f^*$. 
\item \label{it:biconj}  Let $g\in L^\infty (E_1)$.
  If $f^{**} = f$, then we have $(fg)^* = f^* g^*$. 
   
       \item  \label{it:integral-pre} If the pair $(f_1,f_2)$ is pre-conjugate, then 
         $\int_{E_1} f_1\,  \rd\mu_1 = \int_{E_2} f_2\, \rd\mu_2$.  
\item \label{it:preconjfg}
Let  $g_i\in L^\infty (E_i)$, $i=1,2$  If the pair $(f_1,f_2)$ is conjugate and the pair $(g_1,g_2)$ is pre-conjugate,
  then the pair $(f_1g_1, f_2g_2)$ is pre-conjugate.
\end{enumerate}
\end{lemma}

\begin{proof}
 Since $(f^{**}, f^*)$ is conjugate by Lemma~\ref{lem:conj1}, we deduce that $f^{***}=f^*$ by
 definition of conjugation.
 To             prove~\ref{it:biconj}              note             that
 $(fg)^*(Z_2)  =   \esp{f(Z_1)g(Z_1)\given\common}$,  but   $f(Z_1)$  is
 $\common$-measurable  as  $f^{**}=f$, so  we  may  pull it  out.  Since
 $f(Z_1)  = f^*(Z_2)$  and $\esp{g(Z_1)\given\common}  = g^*(Z_2)$,  the
 result follows.

If $(f_1, f_2)$ is pre-conjugate, we have $
    \esp{ f_1(Z_1) \given \common} = f^*_1(Z_2)
= f_2^*(Z_1)   = \esp{f_2(Z_2)  \given \common}$; then   take the expectation to get~\ref{it:integral-pre}.
 Point~\ref{it:preconjfg} is a direct consequence of
 Point~\ref{it:biconj} and Lemma~\ref{lem:conj0}. 
 \end{proof}

 \subsection{Proof of Lemma~\ref{lem:extended} and a key lemma}
 \label{sec:proof-extended}

 Let us first  recall an elementary  result on conditional
independence.  The
random variables  we consider  are  defined on a probability  space, say
$(\Omega_0, \cf_0, \P)$.
 Let $\ca$,  $\cb$ and $\ci$ be  sub-$\sigma$-fields  of
$\cf_0$.  We recall  that $\ca$ and $\cb$  are conditionally independent
given  $\ci$, denoted by $\ca\independent_\ci \cb$,
if $\prb{A \cap  B \given \ci}=\prb{A\given \ci}  \prb{B \given \ci}$  for all
$A\in   \ca$  and   $B\in  \cb$.  According   to  \cite[Theorem
8.9]{Kal21}, if $\ci\subset \ca\cap\cb$, the conditional independence
$ \ca \independent_\ci  \cb$
holds if and only if 
\begin{equation}\label{eq:kall}
  \esp{W|\, \cb}=\esp{W|\, \ci} \text{ for any nonnegative $\ca$-measurable variable $W$.}
\end{equation}

We start by a probabilistic result. 
\begin{lemma}
  \label{lem:prob_key}
  Let $E_1$ and $E_2$ be  coupled, and $\bE_1$ and $\bE_2$ coupled
  through the extended coupling
  $(\bZ_1,\bZ_2) = ((X_1,Y_1),(X_2,Y_2))$. Let  $\common =
  \sigma(\bZ_1)\cap\sigma(\bZ_2)$, $\common_X = \sigma(X_1)\cap\sigma(X_2)$ and
  $\common_Y = \sigma(Y_1)\cap\sigma(Y_2)$.
  \begin{enumerate}[(i)]
  \item
    \label{it:cond_ind}
      The following conditional independence holds:
\(
        \sigma(X_1,X_2) \independent_{\common_X} \common
\) and $\sigma(Y_1,Y_2)\independent_{\common_Y} \common$. 
\item \label{it:KV}
  Let $K,V$ be nonnegative random variables. 
  If $K$ is $\common$-measurable and $V$ is $\sigma(Y_1,Y_2)$-measurable,
  then we have:
  \begin{equation}
    \label{eq:prob_key}
    \esp{KV\given X_i} = \esp{ K \, \esp{V\given\common} \given
      \common_X} = \esp{K\,  \esp{V\given\common}\given X_i},
    \quad i = 1,2. 
  \end{equation}
\end{enumerate}
\end{lemma}

\begin{proof}
  By~\eqref{eq:kall} the  first independence  in Point~\ref{it:cond_ind}
  holds  if, for  any $\common$-measurable  nonnegative random  variable
  $W$,  $\esp{W\given\common_X} =  \esp{W\given  X_1,X_2}$.  Let $W$  be
  $\common$-measurable and nonnegative; so  $W = \phi(X_1,Y_1)$ for some
  function  $\phi$.  Let  $W'  =  \esp{W\given  X_1,X_2}$.  Since
  $\sigma(X_1,X_2)\independent_{\sigma(X_1)} \sigma(X_1,Y_1)$,
\[
    W' = \esp{\phi(X_1,Y_1)\given X_1,X_2} 
              = \esp{\phi(X_1,Y_1)\given X_1} = \esp{W\given X_1}.
\]
Therefore the  random variable $W'$  is $\sigma(X_1)$-measurable.
By  symmetry, it  is also  $\sigma(X_2)$-measurable,  so it  is in  fact
$\common_X$-measurable.  Therefore, by the tower property, we get:
\[
  \esp{W\given X_1,X_2} = W' = \esp{W'\given \common_X} =
  \esp{W\given \common_X}.
\]
   This proves the first point.

   Since $K$ is $\common$-measurable we  may write it as $K=k(X_1,Y_1)$;
   similarly          $V          =          v(Y_1,Y_2)$.          Since
   $\sigma(X_1,Y_1,Y_2)\independent_{\sigma(X_1)}  \sigma(X_1,X_2)$, and
   since $KV$ is $\sigma(X_1,Y_1,Y_2)$-measurable, we get:
   \[
     \esp{KV|X_1}  = \esp{KV|X_1,X_2}.
   \]
   Let  $W$  denote this  random
  variable.  The same argument applied with the conditional independence
  $\sigma(X_2,Y_1,Y_2)    \independent_{\sigma(X_2)}    \sigma(X_1,X_2)$
  yields                                                   symmetrically
  $\esp{KV\given X_2} =  \esp{KV\given X_1,X_2} = W$.  In particular $W$
  is  measurable  with respect  to  both  $X_1$  and  $X_2$, so  $W$  is
  $\common_X$-measurable.   Using  the  tower  property  of  conditional
  expectations      with      $\common_X\subset\sigma(X_1,X_2)$      and
  $\common_X\subset\common$,    and    the     fact    that    $K$    is
  $\common$-measurable, we get:
 \[
    W = \esp{W\given \common_X} = \esp{KV \given \common_X} = \esp{\esp{KV\given \common} \given \common_X} 
 = \esp{K\esp{V\given\common} \given \common_X}.
  \]  
This proves the first equality  of~\eqref{eq:prob_key} for $i=1$ and
then for $i=2$ by symmetry. Set  $V' = \esp{V\given \common}$ which is
$\common_Y$-measurable  and then $\sigma(Y_1, Y_2)$-measurable. Then
apply the first equality  of~\eqref{eq:prob_key} with $V$ replaced by
$V'$ to get the second equality  of~\eqref{eq:prob_key}. 
  The proof is then complete. 
\end{proof}


The fact that conjugacy behaves well 
on extended spaces is  now easy to establish.

\begin{proof}[Proof of Lemma~\ref{lem:extended}]
  Let $\phi(X_1,Y_1) = f_1(X_1)$. Since $\sigma(X_1) \independent_{\common_X}\common$
  by the first point of Lemma~\ref{lem:prob_key}, 
  we get by~\eqref{eq:kall} that 
  \( \esp{\phi(X_1,Y_1) \given \common} = \esp{f_1(X_1) \given\common}
    = \esp{f_1(X_1)\given \common_X}.\)
\end{proof}

The next lemma is the key to all our main results.  For a probability
space $(E, \ce,\mu)$, say that a kernel $\kk$ on $E$ is \emph{nice}
if $\kk\in L^1(E^2)$ and satisfies
$\int _E \kk(\cdot, y)\,\mu(\rd y)\in L^\infty (E)$. For a nice kernel $\kk$  we define the
bounded operator $T_\kk$ on $L^\infty (E)$ by
$T_\kk(g)=\int _E \kk(\cdot, y) g(y) \,\mu(\rd y)$.
\begin{lemma}[Operator defined by conjugated kernels]
  \label{lem:conj_kernels}
  Let  two spaces  $E_1$ and  $E_2$ be  coupled through  $\pi$. If  the
  nice kernel  $\kk$ on  $E_1$ satisfies $\kk  = \kk^{**}$  (for the
  extended coupling)  and if  $v\in L^\infty (E_1)$,  then $\kk^*$  is a
  nice kernel on $E_2$, $v^*\in L^\infty (E_2)$ and:
 \begin{equation}
   \label{eq:TT*}
   T_{\kk}(v) = T_{\kk}(v)^{**} = T_{\kk}(v^{**})
   \quad\text{and}\quad
       T_{\kk}(v)^* = T_{\kk^*}(v^*).
 \end{equation}
\end{lemma}

\begin{proof}
  Let $(X_1,Y_1,X_2,Y_2)  \sim \extended$ denote the  extended coupling.
  Let $\kk$ be a nice kernel on $E_1$ such that $\kk = \kk^{**}$.  
As $\kk=\kk^{**}$, we deduce from Lemma~\ref{lem:conj1}~\ref{it:ff**}
that $(\kk, \kk^*)$ is conjugate and by Lemma~\ref{lem:conj0} that a.s.\
$\kk(X_1, Y_1)=\kk^*(X_2, Y_2)$ and that this random variable is  $\common$-measurable.

Let  $v\in  L^\infty  (E_1)$.   The  function  $T_{\kk}(v)$  admits  the
probabilistic representation:
\[
  T_{\kk}(v)(X_1) = \esp{ \kk(X_1,Y_1) v(Y_1) \given X_1}.
\]
We  apply Lemma~\ref{lem:prob_key}~\ref{it:KV}  with $K  = \kk(X_1,Y_1)$
and    $V    =    v(Y_1)$    to   get    that    $T_{\kk}(v)(X_1)$    is
$\common_X$-measurable, and  by Lemma~\ref{lem:conj1}~\ref{it:ff**} that
$  T_{\kk}(v)   =  T_{\kk}(v)^{**}$.  This  gives   the  first  equality
of~\eqref{eq:TT*}.

It is  obvious that if $v^*\in L^\infty
  (E_2)$. Using  the definition of the conjugate, and then
 $\esp{V\given\common} = v^*(Y_2)$ from 
Lemma~\ref{lem:extended} and  
Equation~\eqref{eq:prob_key}, we obtain:
\[
  T_{\kk}(v)^*(X_2) 
 = \esp{\kk(X_1,Y_1)\, v(Y_1) \given \common_X}
=  \esp{\kk^*(X_2,Y_2)\,  v^*(Y_2) \given X_2}
  = T_{\kk^*}(v^*)(X_2).
\]
Taking $v=\un$, we deduce that $T_{\kk^*}(\un)=T_\kk(\un)^*$ belongs
to $L^\infty (E_2)$, thus $\kk^*$ is a nice kernel on $E_2$ and
$T_{\kk^*}$ is a bounded operator on $L^\infty (E_2)$.  We have also
proven that $T_\kk(v)^*=T_{\kk^*}(v^*)$ which is the last equality
of~\eqref{eq:TT*}.  Using this equality again with $\kk$ and $v$ 
replaced by $\kk^*$ and $v^*$, we obtain that
$ T_\kk(v)^{**}=T_{\kk^*}(v^*)^*=T_{\kk^{**}}(v^{**})=T_\kk(v^{**})$,
which is the second equality of~\eqref{eq:TT*}.
\end{proof}

\subsection{Proof of the main result, Theorem~\ref{thr:main}}
\label{sub:proof_main}

\paragraph{The spectrum and effective reproduction number.}
We   prove    the   first    item   of    Theorem~\ref{thr:main}.    
Recall  the
    spectral radius of a bounded operator  is the maximal modulus of its
    complex   eigenvalues Set     $\kk_i=k_i/\gamma_i$ for  $i=1,2$.  Notice the bounded operators  $T_{\kk_i}$ on
$L^\infty  (E_i)$  and $\Tinf_{\kk_i}$  on  $\cl  (E_i)$ have  the  same
spectrum  and  thus   the  same  spectral  radius   and  more  generally
$R_e(\eta_i)=\rho(T_{\kk_i  \eta_i})$  for  $\eta_i\in  \Delta_i$.   For
simplicity,  write  $\kk=\kk_1$  and   thus,  as  $(\kk_1,\kk_2)$  is  a
conjugate pair, $\kk^*=\kk_2$ and $\kk^{**}=\kk$.

Let $\eta\in \Delta_1$ and $\lambda$ be a non-zero eigenvalue of
    $T_{\kk \eta}$ associated with an eigenvector
    $v\in L^1(E_1)$.
    By definition, we have:
    \[
      \lambda v = T_{\kk\eta}(v) = T_{\kk}(\eta
      v).
    \]
Thanks to  the first two equalities in~\eqref{eq:TT*} of
Lemma~\ref{lem:conj_kernels},  the function $\lambda v$ is equal to its
biconjugate (that is, the pair  $(v, v^*)$ is conjugate) and  $
  \lambda v = T_{\kk}( (\eta v)^{**})$.


Assume the pair $(\eta, \eta_2)$ is  pre-conjugate. By
Lemma~\ref{lem:conj}~\ref{it:preconjfg}, the pair $(\eta v, \eta_2
v^*)$ is pre-conjugate, and thus $(\eta_2 v^*)^*=(\eta v)^{**}$. 
  Then, using Lemmas~\ref{lem:conj_kernels} and~\ref{lem:conj}~\ref{it:biconj}, we get:
    \begin{align*}
      T_{\kk^*}(\eta_2 v^*)
      = T_{\kk^*}((\eta_2 v^*)^{**}) 
 = T_{\kk}( (\eta_2 v^*)^*) ^*
      = T_\kk ( (\eta v)^{**})^*
        = \lambda v^*.
    \end{align*}  
    Since $v^{**}=v\neq \zero$, the function $v^*$ is non-zero and it is therefore
    an   eigenvector    of   $T_{\kk^*\eta_2}$,   associated    to   the
    eigenvalue~$\lambda$.  By symmetry we deduce that the spectrum up to
    $\{0\}$ of~$T_{\kk  \eta}$ and~$T_{\kk^*\eta_2}$ coincide,  and thus
    their spectral radius are equal.  This proves Point~\ref{it:th-Re}.



  \paragraph{The equilibria.}
  Let us now prove the first part of Point~\ref{it:th-I} on the
  equilibria are conjugate.  
  Let $g\in \cl(E_1)$ be an equilibrium of the model $\param_1$. 
    Since~$F_{\eta}(g)=0$, we have:
 \[ 
    g = \frac{\Tinf_{k_1}(\eta g)}{\gamma_1 + \Tinf_{k_1}(\eta g)}\cdot
\] 
By   Lemma~\ref{lem:conj_kernels},  seeing   $g$   as   an  element   of
$L^\infty  (E_1)$,  we  get  that  $T_{k_1}(\eta g)$  is  equal  to  its
biconjugate.  Since $\mu_1$-a.e. $\gamma_1^{**}  = \gamma_1$ , we easily
deduce using Lemma~\ref{lem:conj}~\ref{it:biconj} that:
  \[
    g^* = \frac{T_{k_1}(\eta g)^*}{\gamma_1^* + T_{k_1}(\eta
      g)^*}
    \quad\text{and then}\quad
       g^{**} = \frac{T_{k_1}(\eta g)}{\gamma_1 + T_{k_1}(\eta
      g)},
  \]
  that  is, $\mu_1$-a.e.\  $g^{**} =  g$. So  $(g, g^*)$  is
  conjugate.    By  Lemma~\ref{lem:conj}~\ref{it:preconjfg},   the  pair
  $(\eta    g,    \eta_2    g^*)$    is    pre-conjugate,    and    thus
  $(\eta g)^* = (\eta_2 g^*)^{**}$. 
We get, using  Lemma~\ref{lem:conj_kernels} for the first and last equalities:
  \[
    T_{k_1} (\eta g)^* = T_{k_2}((\eta g)^*) =T_{k_2}((\eta_2 g^*)^{**}
    ) = T_{k_2}(\eta_2 g^*).
  \]
Notice that if  $\mu_2$-a.s.\ $f=h$
  then $\Tinf_{k_2}(f)=\Tinf_{k_2}(g)$, so that $\Tinf_{k_2}(\eta_2 g^*)$
  is a well defined element of $\cl(E_2)$.
  Thus  defining $g_2\in \cl(E_2)$ by:
   \[
    g_2 = \frac{\Tinf_{k_2}(\eta_2 g^*)}{\gamma_2 + \Tinf_{k_2}(\eta_2 g^*)},
   \]
   we get that $\mu_2$-a.e.\ $g_2=g^*$ and that $F_{\eta_2}(g_2)=0$.  In
   other   words,~$g_2$  is   an   equilibrium  for   the  model   given
   by~$\param_2$  when using  the  vaccination  strategy $\eta_2$,  and,
   seeing $g_i$ as  an element of $L^1 (E_i)$, the  pair $(g_1, g_2)$ is
   conjugate. This proves the first part of  Point~\ref{it:th-I}.

   \paragraph{The fraction of infected individuals $\I$.}
   We  now  prove  that  $\I[\param_1](\eta_1)  =  \I[\param_2](\eta_2)$
   whenever  the  pair  $(\eta_1,\eta_2)$ is  preconjugate.   We  assume
   without loss of  generality that $R_0[\param_1]=R_e[\param_1](\un)>1$
   which is equivalent  to~$R_0[\param_2]=R_e[\param_2](\un) >1$, thanks
   to Theorem~\ref{thr:main}~\ref{it:th-Re} as the  pair $(\un, \un)$ is
   conjugate and thus  pre-conjugate.  Let~$g_1 = \mathfrak{g}_{\eta_1}$
   be the  maximal equilibrium for  the model~$\param_1$ when  using the
   vaccination strategy $\eta_1$. By the previous result there exists an
   equilibrium $g_2$  for SIS  model $\param_2$ such  that $\mu_2$-a.s.\
   $g_2 =  g_1^*$. Let us  now prove that it  is the maximal  one. Since
   $(1-g_2)     =    (1-     g_1)^*$    in     $L^1(E_2)$,    we     get
   $R_e[\param_1](1-g_1)     =    R_e[\param_2](1-g_2)$,     again    by
   Theorem~\ref{thr:main}~\ref{it:th-Re}.        Since~$R_0[\param_1]>1$
   and~$g_1$  is  the  maximal  equilibrium  for~$\param_1$,  we  deduce
   from~\cite[Proposition~5.5]{ddz-hit}  that  the vaccination  strategy
   associated   to~$g_1$  is   critical,   that  is,   $R_e[\param_1](1-
   g_1)=1$.  Since $g_2$  is  an  equilibrium for~$\param_2$  satisfying
   $R_e[\param_2](1-        g_2)=1$,        we       deduce        using
   again~\cite[Proposition~5.5]{ddz-hit} that~$g_2$ is  also the maximal
   equilibrium   for~$\param_2$.   Using   Point~\ref{it:preconjfg}  of
   Lemma~\ref{lem:conj},      we      deduce     that      the      pair
   $(g_1   \eta_1,  g_2   \eta_2)$  is   pre-conjugate  and   then  from
   Point~\ref{it:integral-pre}                therein               that
   $\I_1(\eta_1)  =  \int_{E_1}  \eta_1  g_1 \,  \rd\mu_1  =  \int_{E_2}
   \eta_2g_2\,  \rd\mu_2  =  \I_2(\eta_2)$.   This  ends  the  proof  of
   Point~\ref{it:th-I}.

   \paragraph{Proof of  Point~\ref{it:main}.}
   Thanks to Points~\ref{it:th-Re} and~\ref{it:th-I},  it is enough to
   check that  $C[\param_1](\eta_1) = C[\param_2](\eta_2)$  whenever the
   pair   $(\eta_1,\eta_2)$   is    pre-conjugate.    Since   the   pair
   $(\costad_1,\costad_2)$ is conjugate, this is a direct consequence of
   Points~\ref{it:integral-pre}        and~\ref{it:preconjfg}       from
   Lemma~\ref{lem:conj}.

\printbibliography

\end{document}